\documentclass[twoside]{IEEEtran}

\usepackage{amssymb,amsmath,amsthm,txfonts}

\usepackage{graphicx}
\usepackage{authblk}
\usepackage{cite}
\usepackage{subcaption}
\usepackage{url}
\usepackage{multirow}
\usepackage{setspace}
\usepackage{wrapfig}
\usepackage{enumerate}
\usepackage{algorithm}
\usepackage{algpseudocode}
\usepackage{hyperref}
\hypersetup{
  colorlinks   = true,    
  urlcolor     = blue,    
  linkcolor    = blue,    
  citecolor    = red,     
}
\DeclareMathOperator*{\argmin}{argmin}

\newtheorem{theorem}{{\bf Theorem}}
\newtheorem{corollary}{{\bf Corollary}}

\newtheorem{example}{\bf Example}

\newtheorem{definition}{{\bf Definition}}
\newtheorem{lemma}{\bf Lemma}
\newtheorem{assumption}{\bf Assumption}

\begin{document}
%

%

\title{Prediction in Online Convex Optimization for Parametrizable Objective Functions}
\author[1]{Robert Ravier}
\author[1]{Vahid Tarokh}
\affil[1]{Department of Electrical and Computer Engineering, Duke University}

\maketitle
\begin{abstract}
Many techniques for online optimization problems involve making decisions based solely on presently available information: fewer works take advantage of potential predictions. In this paper, we discuss the problem of online convex optimization for parametrizable objectives, i.e. optimization problems that depend solely on the value of a parameter at a given time. We introduce a new regularity for dynamic regret based on the accuracy of predicted values of the parameters and show that, under mild assumptions, accurate prediction can yield tighter bounds on dynamic regret. Inspired by recent advances on learning how to optimize, we also propose a novel algorithm to simultaneously predict and optimize for parametrizable objectives and study its performance using simulated and real data.
\end{abstract}

\section{Introduction}
\label{sec:Intro}

Online convex optimization (OCO) has received significant attention in recent years due to its wide range of applicability. These applications include ad selection \cite{hazan2016introduction}, video streaming \cite{joseph2012jointly}, power scheduling \cite{narayanaswamy2012online} among many others. We refer the reader to \cite{shalev2012online,hazan2016introduction} for a more rigorous introduction, but we briefly summarize below.

The typical OCO scenario can be modeled as the following game. At time $t-1,$ the player must pick a candidate point $x_{t}$ for belonging to some constraint set $\mathcal{X}.$ At time $t,$ the true convex loss function $f_t(\cdot)$ is revealed, and the player suffers a loss $f_{t}(x_{t}).$ This continues for a total of $T$ time steps. The general goal is to find an algorithm that performs well with respect to some notion of regret. Historically, much focus has been given towards the static regret, i.e. performance with respect to the optimal fixed point in hindsight:

\begin{equation}\label{eq:Static}
{\bf{Reg}}_{S}(\{x_{t}\}) := \sum_{t=1}^{T}f_{t}(x_{t}) - \min_{x^{\ast} \in \mathcal{X}} \left( \sum_{t=1}^{T}f_{t}(x^{\ast}) \right)
\end{equation}

where $\{x_{t} \}$ is the sequence of moves played. It has been well-established that algorithms can achieve sublinear \cite{zinkevich2003online} and even logarithmic regret under suitable assumptions \cite{hazan2007logarithmic}. The same cannot be said of the measure of regret of interest in our paper, dynamic regret (sometimes called {\emph{restricted dynamic regret}} \cite{zhang2018adaptive} or {\emph{tracking regret}} \cite{hall2013dynamical}). Dynamic regret measures the performance with respect to the optimal values of the function at each time, i.e.:

\begin{equation}\label{eq:Dynamic}
{\bf{Reg}}_{D}(\{x_{t}\}) := \sum_{t=1}^{T}\left( f_{t}(x_{t}) - \min_{x_{t}^{\ast} \in \mathcal{X}} f_{t}(x_{t}^{\ast}) \right)
\end{equation}

OCO algorithms have been more frequently analyzed by looking at their dynamic regret \cite{jadbabaie2015online,mokhtari2016online, yang2016tracking,zhang2017improved,li2018using}. The analysis of these algorithms does not focus on sublinear regret (as this is impossible to achieve in general \cite{yang2016tracking}) but rather focuses on bounding performance in terms of different regularities depending on the specific algorithm employed.

Somewhat surprisingly, few of the algorithms for OCO explicitly make use of predictions of future objective functions and gradients. Indeed, OCO has historically been viewed from an adversarial lens. This is perhaps too conservative for many scenarios. For example, in applications related to power allocation, frequently past information concerning usage is indicative of the future. Leveraging accurate predictions could better assist in many scenarios where online optimization techniques are utilized. However, the effect of accurate predictions of relevant data points on the performance of such online algorithms is generally not clear. Some algorithms give bounds on regret in terms of the accuracy of blackbox predictions of functions or gradients  \cite{jadbabaie2015online,mohri2016accelerating}, but it is not immediately obvious as to how one can get these from data. Other methods assume that accurate predictions are known throughout the duration of the scenario (e.g. \cite{li2018using}), or assume very particular structure of the resulting accuracy or potentially involve unrealistic assumptions such as fully optimizing a function at every point and time (e.g. \cite{chen2016using}). We would like to address these issues in our paper for a large class of objective functions.

To this end, we make the observation that in many situations, the specific {\emph{form}} of the objective function is known and fixed throughout time. More specifically:

\begin{definition} \label{def:Parametric}
Let $f(\cdot,\cdot): \mathbb{R}^{n} \times \mathbb{R}^{m} \to \mathbb{R}$ be a function that is convex in the first argument. An optimization problem is {\bf{parametric}} if it is of the form

$$\min_{x \in \mathcal{X}} f(x,\theta)$$

\noindent where $\mathcal{X} \subseteq \mathbb{R}^{n}$ is a closed convex set for some $\theta$ in a parameter space $\Theta \subset \mathbb{R}^{m}.$
\end{definition}

This is a general form of optimization problem present in predictive optimization problems \cite{chen2016using,ito2018unbiased}. Throughout the remainder of the paper, we assume that we are interested in OCO problems where the cost functions $f_{t}$ are of the form $f_{t}(x) = f(x,\theta_{t}).$ We note that many important objective functions of theoretical and practical interest are encompassed by this specific form.

\begin{example}\label{ex:FunTimeSeries}
Let $g_{1}(x),...,g_{m}(x)$ be convex functions with domain $\mathcal{X} \subseteq{R}^{n}$ and let $\Theta = \Delta^{m}$ be the standard unit $m$-simplex. Then

$$f(x,\theta_{t}) := \theta^{1}_{t}g_{1}(x) + \cdot \cdot \cdot + \theta^{m}_{t} g_{m}(x)$$

is a {\emph{functional time series}}.
\end{example}

\begin{example}\label{ex:Markowitz}
For a collection of assets $A,$ let $\mu_{t}$ and $\Sigma_{t}$ be their corresponding sample mean and covariance, and let $\lambda_{t} > 0.$ For $\mathcal{X} = \Delta^{|A|}$ the unit simplex on the number of assets $|A|,$ the Markowitz optimal portfolio \cite{markowitz1952portfolio} with respect to $[\mu_{t},\Sigma_{t},\lambda_{t}]$ is the $\argmin$ over $X$ of the following function

\begin{equation}\label{eqn:Markowitz}
f(x,[\mu_{t},\Sigma_{t},\lambda_{t}]) = x^{T}\Sigma_{t}x - \lambda_{t} x^{T} \mu_{t}
\end{equation}
\end{example}

For parametric optimization problems, prediction of objective functions and relevant quantities reduces to prediction of parameters, a much more well-studied though still difficult problem (e.g. \cite{brockwell2002introduction,box2015time}).

The main results of our paper are summarized as follows.

\begin{itemize}
\item{We show that, under mild regularity assumptions, gradient descent using predicted objectives for parametric optimization problems as defined in Definition \ref{def:Parametric} can improve the dynamic regret over standard online gradient descent provided sufficient accuracy in predicted values. The method of proof of our dynamic regret bounds is general enough to extend to cases of where a descent algorithm yields a contraction, i.e. we have some inequality of the form

\begin{equation}\label{eq:Contract}
||x_{t+1}-x_{t+1}^{\ast}|| \leq \rho ||x_{t}-x_{t+1}^{\ast}||
\end{equation}
for some $0 < \rho < 1.$}
\item{We provide a meta-learning algorithm called SMAD, inspired by recent innovations in learning how to optimize, that simultaneously learns the optimal parameter prediction process from a collection of models while performing descent.}

\end{itemize}

The remainder of the paper is summarized as follows. In Section \ref{sec:Prelim} we further detail preliminary details and assumptions needed for the remainder of the paper. In Section \ref{sec:Theo} we detail our theoretical results concerning the performance of predictive online gradient descent. In Section \ref{sec:SMAD} we detail our meta-learning algorithm for simultaneous modeling and descent. In Section \ref{sec:Numer} we give numerical simulations to backup our intutition and evaluate our algorithm's performance. In Section \ref{sec:Conc} we make concluding remarks.
\section{Preliminaries}\label{sec:Prelim}

\subsection{Regularities for Dynamic Regret}
As discussed earlier, dynamic regret bounds for algorithms focus on various regularities of the OCO problem of interest. These regularities generally focus not on algorithmic decisions but on properties of the elements of the problem outside of the algorithm's control. We briefly a number of these quantities. One of the more prevalent regularities, notably appearing in \cite{zinkevich2003online,mokhtari2016online} is the path length of the sequence of optimal points: if $x_{t}^{\ast} = \argmin_{x \in \mathcal{X}} f_{t}(x)$ for $f_{t}$ convex, then

$$\mathcal{P}^{\ast} := \sum_{t=1}^{T-1} ||x_{t}^{\ast}-x_{t+1}^{\ast}||$$

Other regularities of interest include the {\emph{squared}} path length introduced in \cite{zhang2017improved}, functional variation \cite{besbes2015non} and the gradient variation \cite{chiang2012online}, which are measurements that depend on the sup norm of the differences between the functions and their gradients between times $t-1$ and $t.$ 

More directly relevant to our discussion are what we term {\emph{prediction regularities}}. These are not regularities as above in the sense that they are within an algorithm's purview. Nevertheless, they have demonstrated importance for certain regret bounds. \cite{jadbabaie2015online} introduces a squared predictive gradient regularity

$$\sum_{t=1}^{T} ||M_{t}-\nabla_{x}f_{t}(x_{t})||^2$$

\noindent where $M_{t}$ is a prediction of $\nabla_{x}f_{t}(x_{t})$ prior to it becoming revealed. We will consider predictive regularities consistent with the parametric optimization framework we outlined earlier. For $f_{t}(x) = f(x,\theta_{t}),$ we consider the {\emph{parameter prediction regularity}}

$$P^{\theta} := \sum_{t=2}^{T} ||\theta_{t}-\hat{\theta_{t}}||$$

This quantity measures cumulative prediction error respect to the parameters (hence objective functions) and will be important to our subsequent analysis. One can also similarly introduced a squared parameter prediction regularity in a manner analogous to the squared path length, obtained by squaring the norms in the above term, but we do not pursue this. 

\subsection{Theoretical Assumptions}

We now detail the theoretical assumptions needed for the remainder of the paper. The first few assumptions are standard for dynamic regret analysis in OCO and can be seen in, for example, \cite{mokhtari2016online,zhang2017improved}. Recall that all functions we consider will be of the form $f_{t}(x) = f(x,\theta_{t})$  and that our closed, convex constraint set is given by $\mathcal{X}$ with corresponding projection $\Pi_{\mathcal{X}}$ and the parameter set is given by $\Theta.$

\begin{assumption}\label{assum:Grad}
The function $f(x,\theta)$ is Lipschitz continuous in $x,$ i.e. there exists a constant $G > 0$ such that, for all $x,y \in \mathcal{X}$ and $\theta \in \Theta$

$$|f(x,\theta)-f(y,\theta)| \leq G ||x-y||.$$
\end{assumption}
\begin{assumption}\label{assum:LSmooth}
The function $f(x,\theta)$ is $L$-smooth in $x$, i.e. there exists a constant $G > 0$ such that

The function $f(x,\theta)$ is $\lambda$-strongly convex in $x,$ i.e. there exists a constant $G > 0$ such that

$$f(y,\theta) \leq f(x,\theta) + \nabla_{x}f(x,\theta)^{T}(y-x) + \frac{L}{2} ||y-x||^2$$
\end{assumption}
\begin{assumption}\label{assum:StrongCon}
The function $f(x,\theta)$ is $\lambda$-strongly convex in $x,$ i.e. there exists a constant $G > 0$ such that

$$f(y,\theta) \geq f(x,\theta) + \nabla_{x}f(x,\theta)^{T}(y-x) + \frac{\lambda}{2} ||y-x||^2$$
\end{assumption}

The first two assumptions are common throughout the OCO literature and give upper bounds on the first and second derivatives. Assumption \ref{assum:StrongCon} is a more recent assumption in the OCO literature, first appearing in dynamic regret analysis in \cite{mokhtari2016online}, but is a common assumption for the analysis of descent algorithms like gradient descent \cite{boyd2004convex}. Quadratic functions defined over a compact set are examples of functions satisfying all three assumptions.

For our analysis, we will need an additional regularity assumption concerning the behavior of gradients with respect to $\theta.$
\begin{assumption}\label{assum:LipTheta}
The function $f(x,\theta)$ has Lipschitz continuous $x$-gradients in $\theta,$ i.e. there exists some $C_{\theta} > 0$ such that, for all $\theta_{1},\theta_{2} \in \Theta$ and $x \in \mathcal{X},$ we have

$$|\nabla_{x}f(x,\theta_{1})-\nabla_{x}f(x,\theta_{2})| \leq C_{\theta} ||\theta_{1}-\theta_{2}||.$$
\end{assumption}

It is not hard to check that functional time series (Example \ref{ex:FunTimeSeries}) satisfies Assumption \ref{assum:LipTheta} provided that the sum of the $\nabla_{x}(g_{1}(x)+ \cdot \cdot \cdot + g_{m}(x))$ is bounded. It is similarly easy to see that the Markowitz portfolio function in Example \ref{ex:Markowitz} also satisfies Assumption \ref{assum:LipTheta} when viewing the collection of parameters as a column-stacked vector.

\section{Theoretical Results for Prediction in Descent}\label{sec:Theo}
\begin{algorithm}[tb]

\begin{algorithmic}
\caption{Online Predictive Gradient Descent}\label{alg:GradientDescent}
\State {Input: Step size $\eta > 0,$ and $x_{1} \in \mathcal{X}$}
\For {$t = 1 \to T$}
	\State{Receive parameter $\theta_{t}$}
	\State{Predict $\hat{\theta_{t+1}}$ from $\theta_{t},...,\theta_{1}$}
	\State{Compute $x_{t+1} = \Pi_{\mathcal{X}}(x_{t}-\eta \nabla_{x}f(x,\hat{\theta_{t+1}})$}
\EndFor
\end{algorithmic}
\end{algorithm}
We now analyze gradient descent when incorporating prediction. See Algorithm \ref{alg:GradientDescent} for the pseudocode. The main difference between Algorithm \ref{alg:GradientDescent} and standard online gradient descent is in the prediction step in the form of the parameter prediction. As prediction can mean many different things depending on the situation at hand, we avoid mentioning a particular process at this stage.

In analyzing Algorithm \ref{alg:GradientDescent}, we are specifically interested in bounding the dynamic regret in terms of the path-length expressions $\mathcal{P}^{\ast}$ and $\mathcal{S}^{\ast}$ as well as the parameter prediction regularity $P^{\theta}.$ To this end, we have the following result.

\begin{theorem} \label{thm:Main}
Let Assumptions 1-4 hold. If $\eta \leq 1/L,$ then for a $C_{\eta,\lambda} < 1$ we have the following bound on regret of Algorithm \ref{alg:GradientDescent}:

\begin{equation}
{\bf{Reg}}_{D}(\{x_{t}\}) \leq \frac{G||x_{1}-x_{1}^{\ast}||}{1-C_{\eta,\lambda}} + \frac{GC_{\eta,\lambda}}{1-C_{\eta,\lambda}}\mathcal{P}^{\ast} + \frac{G\eta C_{\theta}}{1-C_{\eta,\lambda}}P^{\theta} 
\end{equation}
\end{theorem}

The following lemma, which we state without proof from \cite{zhang2017improved}, makes the constant $C_{\eta,\lambda}$ in the above theorem more precise.

\begin{lemma}\label{lem:GradCont}
Let $g(x)$ be a $\lambda$-strongly convex function and $L$ smooth with minimum attained at $x^{\ast}$. Then projected gradient descent is a contraction provided that $\eta \leq \frac{1}{L}$: if $\mathcal{X}$ is the constraint set and $\Pi_{\mathcal{X}}$ is the projection onto $\mathcal{X},$ we have

$$|| \Pi_{\mathcal{X}}(v - \eta \nabla g(v))-x^{\ast} || \leq C_{\lambda,\eta}||v-x^{\ast}||$$

\noindent where $C_{\lambda, \eta} :=\sqrt{1-\frac{2 \lambda \eta}{1+\eta \lambda}}.$
\end{lemma} 

The proof of the theorem is similar to other calculations of dynamic regret with a few modifications to accommodate the difference in descent strategy. We briefly summarize the proof methodology and defer the exact details to the Appendix, though the interested reader will also find the discussion in the proof of Corollary \ref{thm:MultGrad} enlightening. Previous bounds on the regret for online gradient descent were built around the fact that the descent direction at time $t$ for guessing $x_{t+1}$ was $\nabla_{x}f_{t}$ and not $\nabla_{x}f_{t+1}.$ However, we are explicitly trying to descend using a prediction of $f_{t+1}.$ Our error in this end will be driven by the quality of our prediction of $\nabla_{x} f_{t+1}.$ Assumption \ref{assum:LipTheta} gives that this is controlled by the quality of the parameter prediction. Standard bounding and rearranging then gives the proof of the theorem.

We discuss the results of the bound. When comparing our result to the results of \cite{mokhtari2016online} and \cite{zhang2017improved}, we notice that an additional multiplicative constant less than 1 appears in front of the path length $\mathcal{P}^{\ast}$ at the cost of the entire $P^{\theta}$ term. If we have perfect prediction, or even near perfect, this allows us to achieve a smaller regret bound, potentially significantly so depending on the exact quantity of $\mathcal{P}^{\ast}.$ For imperfect prediction, there is a tradeoff, and it is possible that previous regret bounds are superior in some instances. We detail some examples.

\begin{example} \label{ex:RungeKutta}
Assume that the actual amount of time between $\theta_{t}$ and $\theta_{t+1}$ is $\Delta t.$ Assume that we know that $\theta_{t}$ satisfies some ordinary differential equation $\dot{\theta}_{t} = V(\theta_{t},t).$ Runge-Kutta methods can be used to numerically integrate the ODE and compute a predicted value $\hat{\theta}_{t+1}$ whose error is of the order $( \Delta t)^{4}$. Sufficiently small values of $\Delta t$ will thus guarantee small contribution from the prediction regularity.
\end{example}

\begin{example} \label{ex:RandomNoise}
Following \cite{chen2016using}, we consider the case that our predictions are accurate up to noise, i.e. $\theta_{t}-\hat{\theta_{t}} = h^{t}_{1}X^{t}_{1}+ \cdot \cdot \cdot h^{t}_{k}X^{t}_{k}$ for each $t$ where the $X^{t}_{i}$ are i.i.d. mean zero sub-Gaussian random variables with variance parameter $\sigma^{2}$ and the $h^{t}_{i}$ are constants. Though this implies that $P^{\theta}$ is a random variable, it is well known that such a linear combination satisfies a high probability bound:

$$\mathbb{P} \left( \left| \sum_{i=1}^{k}h^{t}_{i}X^{t}_{i} \right| > \varepsilon \right)\leq \exp \left( - \frac{\varepsilon^{2}}{\sigma^{2} \left( \sum_{i=1}^{k} (h^{t}_{i})^{2} \right)} \right)$$

This implies that $P^{\theta}$ will also be small with high probability.
\end{example}

\noindent We now consider the case where we may wish to perform multiple descent steps between each iteration. This will inject powers of $C_{\eta,\lambda}$ into the above regret bound, potentially increasing the error bound. Surprisingly, this does not significantly affect the constant in front of $P^{\theta}.$

\begin{corollary}\label{thm:MultGrad}
Let Assumptions 1-4 hold. Consider Algorithm \ref{alg:GradientDescent} with the single gradient descent step replaced by $k$ gradient descent steps If $\eta \leq 1/L,$ thenwe have the following bound on regret of this revised version Algorithm \ref{alg:GradientDescent}:
\begin{equation}
{\bf{Reg}}_{D}(\{x_{t}\}) \leq \frac{G||x_{1}-x_{1}^{\ast}||}{1-C^{k}_{\eta,\lambda}} + \frac{GC^{k}_{\eta,\lambda}}{1-C_{\eta,\lambda}^{k}}\mathcal{P}^{\ast} + \frac{G\eta C_{\theta}}{1-C_{\eta,\lambda}}P^{\theta} 
\end{equation}
\end{corollary}

\begin{proof}
The same idea as the single gradient proof holds, but the estimate is slightly different. Indeed, we can replace the one step gradient descent with a $k$-step gradient to also get a contraction, with $C_{\eta,\lambda}$ being replaced by $C_{\eta,\lambda}^{k}.$ However, as we are computing a gradient descent step with $\hat{\theta_{t}}$ and not $\theta_{t},$ we must use a triangle inequality at every iteration before we can use the contractive estimate.

More precisely, let $z_{t}^{k}$ be the point obtained by using $k$ predicted gradient descent steps from $x_{t-1}$. We can repeat the above analysis to see

\begin{align*}
||z_{t}^{k}-x_{t}^{\ast}|| &= ||z_{t}^{k-1} - \eta \nabla_{x}f(z_{t}^{k-1},\hat{\theta_{t}}) - x_{t}^{\ast}|| \\
&\leq ||z_{t}^{k-1} - \eta \nabla_{x}f(z_{t}^{k-1},\theta_{t})- x_{t}^{\ast} || \\ 
&+ \eta || \nabla_{x}f(z_{t}^{k-1},\theta_{t})-\nabla_{x}f(z_{t}^{k-1},\hat{\theta_{t}})|| \\
&\leq C_{\eta,\lambda} ||z_{t}^{k-1}-x_{t}^{\ast}|| + \eta C_{\theta} ||\theta_{t} - \hat{\theta_{t}}|| \\
\end{align*}

\noindent A routine induction gives
\begin{align*}
||z_{t}^{k}-x_{t}^{\ast}|| &\leq C_{\eta,\lambda}^{k} ||x_{t-1}-x_{t}^{\ast} || + C_{\theta} \eta ||\theta_{t}-\hat{\theta_{t}}|| \sum_{i=0}^{k-1} C_{\eta,\lambda}^{i} \\
&\leq C_{\eta,\lambda}^{k} ||x_{t-1}-x_{t}^{\ast} || + \frac{C_{\theta} \eta((1-C_{\eta,\lambda}^{k})}{1-C_{\eta,\lambda}} ||\theta_{t}-\hat{\theta_{t}}||
\end{align*}
\noindent where the last inequality follows by majorizing the summation in the first inequality by an infinite series and subsequent summation.
\end{proof}
We end this section by investigating the generalizability of our proof to other descent methods. The main idea outlined above is to replace the estimated descent step with the true descent step and then estimate the error by looking at both the contraction of the true descent direction and the prediction error of the descent step. More formally, we consider a general descent algorithm, where the gradient descent step is replaced by $x_{t+1} = \Pi_{\mathcal{X}}(x_{t}-\eta D_{\text{est}})$ for some descent direction $D_{\text{est}}.$ If $D_{\text{true}}$ is the descent direction for the true value of the parameter and $D_{\text{est}}$ is that for the predicted value of the parameter, then for this general version, we have the following estimate:

\begin{align*}
||\Pi_{\mathcal{X}}(x_{t}-\eta D_{\text{est}})-x_{t+1}^{\ast}|| &\leq ||\Pi_{\mathcal{X}}(x_{t}-\eta D_{\text{true}})-x_{t+1}^{\ast}|| \\
&+ ||\Pi_{\mathcal{X}}(x_{t}-\eta D_{\text{true}})-\Pi_{\mathcal{X}}(x_{t}-\eta D_{\text{est}})|| \\
&\leq ||\Pi_{\mathcal{X}}(x_{t}-\eta D_{\text{true}})-x_{t+1}^{\ast}|| \\
&+||\eta D_{\text{true}}-\eta D_{\text{est}}||
\end{align*}
by nonexpansiveness of the projection. Provided that the true descent direction yields a contraction, then with sufficient regularity of the descent directions with respect to the parameter $\theta,$ we can get an analogous result for different regularity assumptions on $f(x,\theta).$ For example, if we were to pick a Newton step, so $D_{\text{true}} = \left( \nabla^{2}_{x} f(x,\theta_{t+1}) \right)^{-1} \nabla_{x}f(x,\theta_{t+1})$ and similarly for $D_{\text{est}},$ we can bound the difference by

\begin{align*}
||D_{\text{true}}-D_{\text{est}}|| &\leq ||A^{-1} \left( \nabla_{x}f(x,\theta_{t+1})-\nabla_{x}f(x,\hat{\theta_{t+1}}) \right)|| \\
&+ ||\left( \left(\nabla^{2}_{x} f(x,\theta_{t+1}) \right)^{-1} - \left(\nabla^{2}_{x} f(x,\hat{\theta_{t+1}}) \right)^{-1} \right) B||
\end{align*}

where $A = \nabla^{2}_{x} f(x,\theta_{t+1})$ and $B = \nabla_{x}f(x,\hat{\theta}_{t+1}).$ Under sufficient regularity of the operator norm of the inverse Hessian of $f,$ we can obtain a similar expression in terms of $P^{\theta}$ as we did for gradient descent in Theorem \ref{thm:Main}. We do not investigate this further, but instead use the above to illustrate that our technique is not particular to the gradient descent step used.

\section{SMAD: Simultaneous Modeling and Descent}\label{sec:SMAD}

The theoretical results in the previous section are rather general, and give regret bounds in terms of the quality of prediction without specifying how this prediction is done in general. Frequently we do not know the true process generating the data we are trying to predict, but instead have a collection of candidate models for which we hope at least one will make accurate predictions. To this end, we would like to develop a practical algorithm that will gradually learn the best data generating model over which to optimize among a predetermined collection of models. 

\begin{algorithm}[tb]

\begin{algorithmic}
\caption{Expert Learning Algorithm for Simultaneous Descent and Data Modeling}\label{alg:Learning}
\State {Input: Parameters $\beta, \eta, \gamma > 0,$ initial data $\theta_{-i_{0}},...,\theta_{0}$}
\State {Output: {\bf{$p_{t}$}} = $[p_{t,1},...,p_{t,N}]$ (predictive distribution over the collection of $N$ models), $x_{t}$ (predictive points for optimization) }
\State {Initialize $w_{1,0}=...=w_{k_0,0} = 1,$ $w_{k,0} = 0$ for $k > k_{0}$}
\State {Initialize $M = k_{0}$}
\For {$t = 1 \to T$}
	\If {Initializing new model}
		\State {Compute $w_{i,t-1} = (1-\beta) w_{i,t-1}$ for $i \leq M$}
		\State {Compute $w_{M+1,t-1} = \beta$}
		\State {Compute $p_{i,t-1} = \frac{w_{i,t}}{\sum_{i=1}^{N} w_{i,t}}$}
		\State {Initialize $x_{t-1}^{M+1} = x_{t-1}$}
		\State {Compute $M = M+1$}
		
	\EndIf
	
	\State {Observe $\theta_{t}$}
	\State {Initialize $x_{t} = 0$}
	\For {$1 \leq i \leq M$}
		\State{Receive $v_{i} = \Pi_{\mathcal{X}}(x_{t-1}^{i}- \eta \nabla_{x}f(x_{t-1}^{i},\hat{\theta_{t}^{i}}))$ from expert $i$}
		\State{Compute $x_{t} = x_{t} + p_{i,t-1}v_{i}$}
		\State{Compute  $l_{i} = \exp[- \gamma f(\mathcal{X}(v_{i}),\theta_{t})]$ }
	\State{Compute $w_{i,t} = p_{i,t-1}l_{i}$}
	\EndFor
	\State {Output $x_{t} = \Pi_{\mathcal{X}}(x_{t})$}
	\State{Compute $w_{i,t} = p_{i,t-1}l_{i}$}
	\State{Compute $p_{i,t} = \frac{w_{i,t}}{\sum_{i=1}^{N} w_{i,t}}$}
\EndFor
\end{algorithmic}
\end{algorithm}

In particular, we follow the ideas of MetaGrad and Ader \cite{van2016metagrad,zhang2018adaptive} and employ an approach based on expert learning. Expert learning has been well studied (see, e.g., \cite{cesa2006prediction}) and is summarized as follows. Each expert corresponds to a particular class of data-generating models (for example, each expert can correspond to the lag of an autoregressive (AR) model). We assume that each expert knows the specific function $f(x,\theta)$ that we are trying to optimize. At time $t-1,$ each expert makes a prediction $\hat{\theta_{t}^{i}}$ as to the future value of $\theta$ and uses this prediction in order to evaluate its own predictive online gradient descent procedure, thus giving a predicted value of $x$ denoted $x_{t}^{i}$. Each expert suffers a loss based on evaluating $f(x_{t}^{i},\theta_{t}),$ the experts are reweighted using a Gibbs posterior update procedure, and the process is repeated until the last time point is reached.

We briefly comment on the conditional statement in our expert descent algorithm. From a practical standpoint, when one wishes to start modeling data with a collection of models, they may not have a sufficient amount of data to reliably use a particular model. For example, if one wishes to model a time series of data with an AR($k$) process via the Yule-Walker equations, one cannot reliably estimate a model if the number of data points is not sufficiently large relative to $k.$ To accommodate for this, we allow the user to specify whether/when additional models are added, and do this by mixing a new model in by reweighting the predictive distribution. When a new model comes online, we propose to initialize its candidate $x$-value at the previous point output by the algorithm. We will make use of these procedures in our numerical examples, but for ease of presentation avoid this in our theoretical analysis.

We present the results of our theoretical analysis below.

\begin{theorem} \label{thm:ExpertLearning}
Assume Assumptions 1-4 hold and that the range of $f(x,\theta),$ $D:= \sup_{x,\theta} f(x,\theta) - \inf_{x,\theta}f(x,\theta)$ is bounded. Then if Algorithm \eqref{alg:Learning} introduces no additional models upon starting, it has regret bound

\begin{align*}
{\bf{Reg}}_{D}(\{x_{t}\}) &\leq \frac{G||x_{1}-x_{1}^{\ast}||}{1-C_{\eta,\lambda}} + \frac{GC_{\eta,\lambda}}{1-C_{\eta,\lambda}}\mathcal{P}^{\ast} \\
&+ \frac{G\eta C_{\theta}}{1-C_{\eta,\lambda}} \min_{i} P^{\theta}_{i} + \frac{D \sqrt{2T}}{4} (1+ \ln N)
\end{align*}
where $P^{\theta}_{i}$ is the parameter prediction regularity for model $i$ and $\min_{i}$ is with respect to the available models.
\end{theorem}
The proof combines standard techniques for evaluating expert learning algorithms as in \cite{cesa2006prediction} along with previous regret analysis for predictive online gradient descent and is reserved for the Appendix.

\section{Numerical Experiments}\label{sec:Numer}
We now detail numerical experiments investigating the efficacy of prediction in OCO for parametric objectives as well as the performance of our objective function. We will primarily compare our prediction related results to standard online gradient descent (OGD), where the descent direction is fully determined by the value of the objective function at the present. From an intuitive perspective, if the process governing the objective function is relatively stationary with small variation between time steps, we would anticipate minimal difference between standard OGD and the predictive version we laid out above. The main differences should arise when there are predictable but significant jumps in the parameter governing the objective function for which a method with close to accurate models will predict reasonably well, whereas OGD will suffer a loss for not catching the jump.

\subsection{OGD versus Prediction with Fixed Model}
We consider the case where we have one reasonable candidate model of the objective function parameter. This first experiment is adapted directly from \cite{mokhtari2016online}. We consider the parametric objective function

\begin{align*}
f_{t}([x_{1},x_{2}]) &= f([x_{1},x_{2}],[a_{t},b_{t},c_{t}] \\
&= 100(x_{1}-a_{t})^{2} + (x_{2}-b_{t})^{2} +c_{t}
\end{align*}

\noindent where $[a_{t},b_{t},c_{t}]$ alternates between $[-100,0,30] + \epsilon_{t}$ and $[100,20,-50]+\epsilon_{t}$ every four iterations, where $\epsilon_{t}$ is three-dimensional Gaussian noise with mean zero and covariance $10I_{3}.$ The constraint set is the disc centered at the origin with radius 50. We compare the performance of OGD with the performance of the following procedure: follow OGD for the first 10 timesteps, then estimate a two dimensional AR(4) model for $[a_{t},b_{t}]$ using the Yule-Walker equations, and use predictive online gradient descent. Both methods will be initalized at $[0,40].$ Following the convention in \cite{mokhtari2016online}, we set the step size for both methods to be $1/200.$ 

\begin{figure}[tb]
\center
\includegraphics[scale=.45]{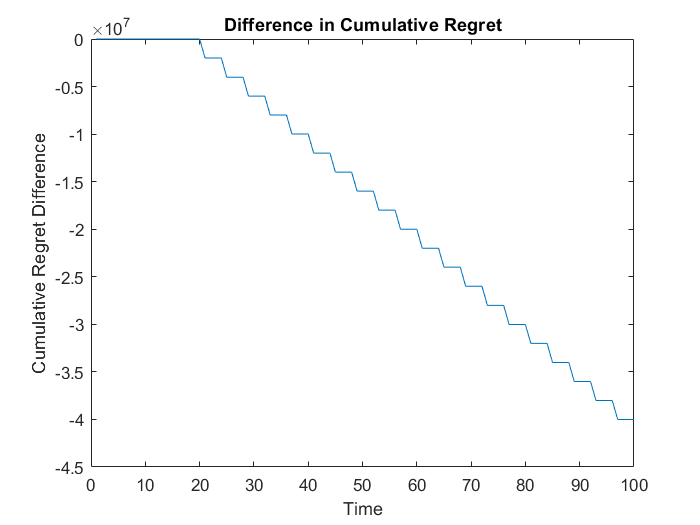}

\caption{Difference in dynamic regret between Predictive OGD and OGD. Lower numbers indicate better performance for Predictive OGD.}
\label{fig:Experiment1}
\end{figure} 

The results of the experiment averaged over fifty repetitions can be found in Figure \ref{fig:Experiment1}. The main validation measure we employ is the difference in cumulative regret between the predictive version of OGD and the standard version of OGD. For this measure, lower values indicate better performance of the predictive method. As expected, the curve remains flat for the first 10 timesteps as the descent method is the same. When the model estimation turns starts, the predictive method begins to outperform the standard OGD as evidenced by the gradually decreasing curve in the figure. The steps on the curve is indicative of the step-like behavior of the cumulative OGD regret as observed in \cite{mokhtari2016online}.

\subsection{Expert Learning on Synthetic Data}
We would now like to test out Algorithm \ref{alg:Learning} in a misspecified setting, i.e. when the model classes do not contain the true objective parameter generating process. To this end, we keep most of the same settings as in the first experiment, but we change the switching process: $[a_{t},b_{t},c_{t}]$ alternates between $[-100,0,30] + \epsilon_{t}$ for four time steps and $[100,20,-50]+\epsilon_{t}$ for six time steps. The models that we use are AR($k$) for $k$ between 1 to 5. We first observe the process for ten time steps before using an AR(1), and then add a new AR model every ten time steps until all five are active. All models are again estimated by the Yule Walker equations. For the other parameters of the algorithm, we use $\beta = 0.2$ and $\gamma = 5 \times 10^{-7}.$

\begin{figure}[tb]
\center
\includegraphics[scale=.45]{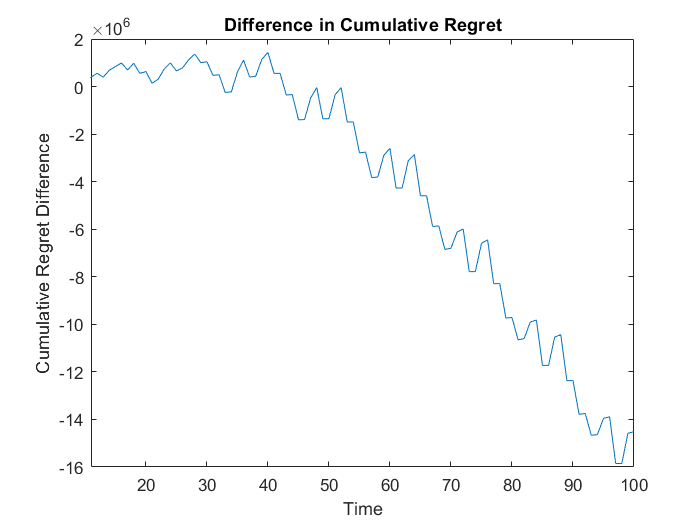}

\caption{Difference in dynamic regret between our expert learning method and OGD. Lower numbers indicate better performance for the expert learning method.}
\label{fig:Experiment2}
\end{figure} 
The results of the experiment averaged over fifty repetitions can be found in Figure \ref{fig:Experiment2}. Initially, the algorithm performs worse than OGD, which is not surprising as the initial models are not close to the parameter process. As higher lag AR models are added and the available models include ones closer to the true process, the performance of the expert learning method improves, eventually becoming the clear favorite over standard OGD.

\subsection{Expert Learning on Financial Data}
As a real world application, we now consider the problem of portfolio optimizationThough there are a number of different ways to construct portfolios in both online and offline settings (see, e.g., \cite{markowitz1952portfolio,cover2011universal,ponsich2013survey}, we restrict our attention to Markowitz portfolio theory as introduced in Example \ref{ex:Markowitz}. To remind the reader, given a collection of assets $z_{1},...,z_{n},$ the Markowitz optimal portfolio allocation is given by:

\begin{equation}
\argmin_{x \in \mathbb{R}^{n}} x^{T}\Sigma_{t}x - \lambda_{t} x^{T} \mu_{t} \\
\end{equation}

$$x_{1},...,x_{n} \geq 0, \sum_{i=1}^{n}x_{i} = 1$$

\noindent where $\Sigma$ represents the covariance matrix for the returns of the assets, $\mu$ is the average return of each asset, and $\lambda > 0$ is a parameter encoding the tradeoff between expected returns and risk; optimizing with $\lambda = 0$ finds the portfolio with the least amount of risk.

In this framework, we consider a hypothetical scenario of a client with a rapidly changing, noisy risk tolerance. In this scenario, a portfolio manager reaches out to his client every month (30 days) with a series of new portfolios, each of which is constructed built by estimation of their client's risk tolerance and subsequent optimizing of Equation \eqref{eqn:Markowitz} given different lookback periods on a given collection of assets to be between 15 and 90 days in increments of 15 for computing relevant means and covariances. The risk tolerance is estimated by a series of autoregressive processes on the risk tolerance with lags between 30 and 180 days in increments of 30\footnote{Risk is only observed every 30 days}.

Unbeknownst to the manager, the client evaluates the portfolio also via Equation \eqref{eqn:Markowitz}, but with a 50 day lookback period and a risk generated by the following process. For the first 240 days, the client's risk parameter is $\max(4 + \varepsilon,0),$ where $\varepsilon$ is Gaussian noise of mean zero and variance 0.64. The remaining risk parameters are generated as follows. Setting $b_{0} = 4,$ we have $\lambda_{t} = \max(b_{t}+\varepsilon_{t},0),$ where $b_{t}$ satisfies for $t \geq 1$:

$$ b_{t+1} = \begin{cases} b_{t} & \text{with probability 0.9} \\ Unif(1,...,20) & \text{with probability 0.1} \end{cases}$$

\noindent and the $\varepsilon_{t}$ are Gaussian with mean zero and variance 0.64. Here, $Unif(1,...,20)$ denotes the discrete uniform distribution on the integers $1,...,20$. As data, we make use of the NYSE dataset used frequently in the portfolio optimization literature, a collection of 36 stock returns taken over a period of 22 years \cite{borodin2004can,li2014online}. We also add a risk-free asset that gives constant, low returns of 1\% every 360 days compounded daily. The goal, as with the synthetic experiment, is to predict the best portfolio of those offered that optimizes the client's objective function without knowing that objective function in the future.

Initially, both methods started with portfolio uniform across the assets. Both method used a descent step size of 0.1 and a learning rate of $\gamma = 50.$ The projection step is performed by setting negative amounts of assets equal to zero before normalizing the percentages of remaining assets. We assume that the learner has 10 months to observe the risk trends before starting, which gives a sufficient amount of data in order to estimate the risk process via the Yule Walker equations. This renders the mixing constant $\beta$ irrelevant to this simulation. Our evaluation period occurs over 150 months.
\begin{figure}[tb]
\center
\includegraphics[scale=.45]{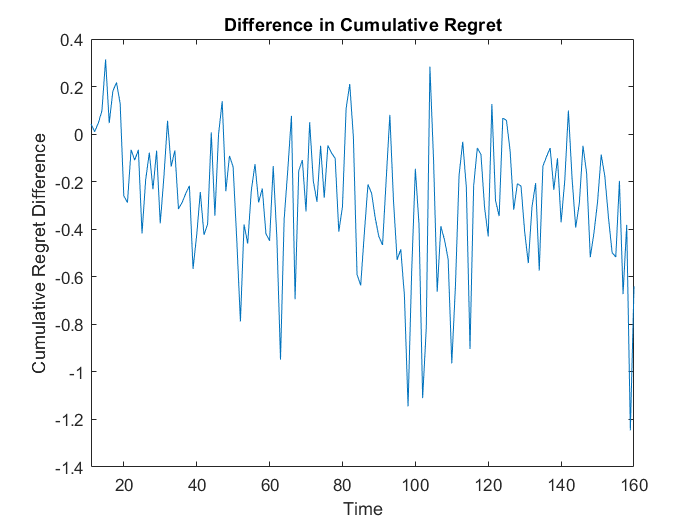}

\caption{Difference in dynamic regret between our expert learning method and OGD. Lower numbers indicate better performance for the expert learning method.}
\label{fig:Experiment3}
\end{figure}

The results of this experiment averaged over 200 times can be seen in Figure \ref{fig:Experiment3}. Unsurprisingly, there is plenty of noise in the resulting curve due to the potentially large fluxuations in the risk parameter. Nevertheless, though there are occasionally spikes for which standard OGD does better, we notice that the difference in cumulative regret between the expert learning and the OGD methods tends to be negative, indicating better performance by our expert learning algorithm.
\section{Conclusion} \label{sec:Conc}

We have discussed the problem of online convex optimization for a wide class of parametric objective functions, for which prediction of parameters subsequently gives us predictions of objective functions. We analyzed a predictive version of online gradient descent and showed that its dynamic regret can improve on currently known bounds provided that prediction of parameters is accurate. We also proposed SMAD, an expert learning-based algorithm that allows us to simultaneously model the parameter process and optimize. We finally showed via numerical examples the power of prediction in OCO, especially in environments where sharp changes can occur, and showed that SMAD can offer better performance than standard online gradient descent in both synthetic and real data settings.

There are a number of directions in which to extend this work. It would be interesting to consider the effect of other smoothness conditions on objective functions, such as self-concordance and semi-strong convexity to see how much improvement we can get on regret bounds as was done in \cite{zhang2017improved}. It would also be interesting to investigate the effect of prediction in optimization when trying to predict parametric {\emph{constraint}} functions, though this will inevitably be challenging due to potential constraint violations from inaccurate predictions. Finally, since predictions can sometimes yield confidence intervals, it would be extremely interesting to explore this problem from the lens of robust optimization, where one focuses on minimizing the maximal possible loss. \cite{ben2009robust}

\section{Acknowledgements}

This work was funded by DARPA grant number FA8650-18-1-7837.


\bibliography{icmlbib}

\begin{thebibliography}{10}

\bibitem{ben2009robust}
A.~Ben-Tal, L.~El~Ghaoui, and A.~Nemirovski.
\newblock {\em Robust optimization}, volume~28.
\newblock Princeton University Press, 2009.

\bibitem{besbes2015non}
O.~Besbes, Y.~Gur, and A.~Zeevi.
\newblock Non-stationary stochastic optimization.
\newblock {\em Operations research}, 63(5):1227--1244, 2015.

\bibitem{borodin2004can}
A.~Borodin, R.~El-Yaniv, and V.~Gogan.
\newblock Can we learn to beat the best stock.
\newblock In {\em Advances in Neural Information Processing Systems}, pages
  345--352, 2004.

\bibitem{box2015time}
G.~E. Box, G.~M. Jenkins, G.~C. Reinsel, and G.~M. Ljung.
\newblock {\em Time series analysis: forecasting and control}.
\newblock John Wiley \& Sons, 2015.

\bibitem{boyd2004convex}
S.~Boyd and L.~Vandenberghe.
\newblock {\em Convex optimization}.
\newblock Cambridge university press, 2004.

\bibitem{brockwell2002introduction}
P.~J. Brockwell, R.~A. Davis, and M.~V. Calder.
\newblock {\em Introduction to time series and forecasting}, volume~2.
\newblock Springer, 2002.

\bibitem{cesa2006prediction}
N.~Cesa-Bianchi and G.~Lugosi.
\newblock {\em Prediction, learning, and games}.
\newblock Cambridge university press, 2006.

\bibitem{chen2016using}
N.~Chen, J.~Comden, Z.~Liu, A.~Gandhi, and A.~Wierman.
\newblock Using predictions in online optimization: Looking forward with an eye
  on the past.
\newblock {\em ACM SIGMETRICS Performance Evaluation Review}, 44(1):193--206,
  2016.

\bibitem{chiang2012online}
C.-K. Chiang, T.~Yang, C.-J. Lee, M.~Mahdavi, C.-J. Lu, R.~Jin, and S.~Zhu.
\newblock Online optimization with gradual variations.
\newblock In {\em Conference on Learning Theory}, pages 6--1, 2012.

\bibitem{cover2011universal}
T.~M. Cover.
\newblock Universal portfolios.
\newblock In {\em The Kelly Capital Growth Investment Criterion: Theory and
  Practice}, pages 181--209. World Scientific, 2011.

\bibitem{hall2013dynamical}
E.~Hall and R.~Willett.
\newblock Dynamical models and tracking regret in online convex programming.
\newblock In {\em International Conference on Machine Learning}, pages
  579--587, 2013.

\bibitem{hazan2007logarithmic}
E.~Hazan, A.~Agarwal, and S.~Kale.
\newblock Logarithmic regret algorithms for online convex optimization.
\newblock {\em Machine Learning}, 69(2-3):169--192, 2007.

\bibitem{hazan2016introduction}
E.~Hazan et~al.
\newblock Introduction to online convex optimization.
\newblock {\em Foundations and Trends{\textregistered} in Optimization},
  2(3-4):157--325, 2016.

\bibitem{ito2018unbiased}
S.~Ito, A.~Yabe, and R.~Fujimaki.
\newblock Unbiased objective estimation in predictive optimization.
\newblock In {\em International Conference on Machine Learning}, pages
  2181--2190, 2018.

\bibitem{jadbabaie2015online}
A.~Jadbabaie, A.~Rakhlin, S.~Shahrampour, and K.~Sridharan.
\newblock Online optimization: Competing with dynamic comparators.
\newblock In {\em Artificial Intelligence and Statistics}, pages 398--406,
  2015.

\bibitem{joseph2012jointly}
V.~Joseph and G.~de~Veciana.
\newblock Jointly optimizing multi-user rate adaptation for video transport
  over wireless systems: Mean-fairness-variability tradeoffs.
\newblock In {\em INFOCOM, 2012 Proceedings IEEE}, pages 567--575. IEEE, 2012.

\bibitem{li2014online}
B.~Li and S.~C. Hoi.
\newblock Online portfolio selection: A survey.
\newblock {\em ACM Computing Surveys (CSUR)}, 46(3):35, 2014.

\bibitem{li2018using}
Y.~Li, G.~Qu, and N.~Li.
\newblock Using predictions in online optimization with switching costs: A fast
  algorithm and a fundamental limit.
\newblock In {\em 2018 Annual American Control Conference (ACC)}, pages
  3008--3013. IEEE, 2018.

\bibitem{markowitz1952portfolio}
H.~Markowitz.
\newblock Portfolio selection.
\newblock {\em The journal of finance}, 7(1):77--91, 1952.

\bibitem{mohri2016accelerating}
M.~Mohri and S.~Yang.
\newblock Accelerating online convex optimization via adaptive prediction.
\newblock In {\em Artificial Intelligence and Statistics}, pages 848--856,
  2016.

\bibitem{mokhtari2016online}
A.~Mokhtari, S.~Shahrampour, A.~Jadbabaie, and A.~Ribeiro.
\newblock Online optimization in dynamic environments: Improved regret rates
  for strongly convex problems.
\newblock In {\em Decision and Control (CDC), 2016 IEEE 55th Conference on},
  pages 7195--7201. IEEE, 2016.

\bibitem{narayanaswamy2012online}
B.~Narayanaswamy, V.~K. Garg, and T.~Jayram.
\newblock Online optimization for the smart (micro) grid.
\newblock In {\em Proceedings of the 3rd international conference on future
  energy systems: where energy, computing and communication meet}, page~19.
  ACM, 2012.

\bibitem{ponsich2013survey}
A.~Ponsich, A.~L. Jaimes, and C.~A.~C. Coello.
\newblock A survey on multiobjective evolutionary algorithms for the solution
  of the portfolio optimization problem and other finance and economics
  applications.
\newblock {\em IEEE Transactions on Evolutionary Computation}, 17(3):321--344,
  2013.

\bibitem{shalev2012online}
S.~Shalev-Shwartz et~al.
\newblock Online learning and online convex optimization.
\newblock {\em Foundations and Trends{\textregistered} in Machine Learning},
  4(2):107--194, 2012.

\bibitem{van2016metagrad}
T.~van Erven and W.~M. Koolen.
\newblock Metagrad: Multiple learning rates in online learning.
\newblock In {\em Advances in Neural Information Processing Systems}, pages
  3666--3674, 2016.

\bibitem{yang2016tracking}
T.~Yang, L.~Zhang, R.~Jin, and J.~Yi.
\newblock Tracking slowly moving clairvoyant: Optimal dynamic regret of online
  learning with true and noisy gradient.
\newblock In {\em International Conference on Machine Learning}, pages
  449--457, 2016.

\bibitem{zhang2018adaptive}
L.~Zhang, S.~Lu, and Z.-H. Zhou.
\newblock Adaptive online learning in dynamic environments.
\newblock In {\em Advances in Neural Information Processing Systems}, pages
  1330--1340, 2018.

\bibitem{zhang2017improved}
L.~Zhang, T.~Yang, J.~Yi, J.~Rong, and Z.-H. Zhou.
\newblock Improved dynamic regret for non-degenerate functions.
\newblock In {\em Advances in Neural Information Processing Systems}, pages
  732--741, 2017.

\bibitem{zinkevich2003online}
M.~Zinkevich.
\newblock Online convex programming and generalized infinitesimal gradient
  ascent.
\newblock In {\em Proceedings of the 20th International Conference on Machine
  Learning (ICML-03)}, pages 928--936, 2003.

\end{thebibliography}
\bibliographystyle{abbrv}
\onecolumn
\begin{appendix}
\subsection{Proof of Theorem \ref{thm:Main}}\label{sec:MainProof}

The dynamic regret upper bound computation is similar to others. Recall that $x_{t}^{\ast}$ denotes the minimizers  From the assumption that $x$-gradients are bounded above by a constant $G$ for all $\theta,$ we have

\begin{equation}\label{eqn:initGradEst}
\sum_{t=1}^{T} f(x_{t},\theta_{t})- f(x_{t}^{\ast},\theta_{t}) \leq G \sum_{t=1}^{T} ||x_{t}-x_{t}^{\ast}|| = G||x_{1}-x_{1}^{\ast}|| + G\sum_{t=2}^{T} ||x_{t}-x_{t}^{\ast}||
\end{equation}

\noindent To bound the sum on the right-hand side, observe by the definition of the algorithm and the triangle inequality, we have:
\begin{align*}
\sum_{t=2}^{T} ||x_{t}-x_{t}^{\ast}||&= \sum_{t=2}^{T} || (\Pi_{\mathcal{X}}(x_{t-1}- \eta \nabla_{x}f(x_{t-1},\hat{\theta_{t}}) -x_{t}^{\ast}|| \\
&\leq \sum_{t=2}^{T} \left(|| \Pi_{\mathcal{X}}(x_{t-1}- \eta \nabla_{x}f(x_{t-1},\theta_{t})) -x_{t}^{\ast}|| + \eta ||\nabla_{x}f(x_{t-1},\theta_{t}) - \nabla_{x}f(x_{t-1},\hat{\theta_{t}})|| \right)
\end{align*}

\noindent We implicitly used the fact that the projection operator is nonexpansive in our setting in the second inequality. So, we have, referencing Lemma \ref{lem:GradCont}:

\begin{align*}
\sum_{t=2}^{T} ||x_{t}-x_{t}^{\ast}||&\leq \sum_{t=2}^{T} \left(||\Pi_{\mathcal{X}}(x_{t-1}- \eta \nabla_{x}f(x_{t-1},\theta_{t}) -x_{t}^{\ast}|| + \eta ||\nabla_{x}f(x_{t-1},\theta_{t}) - \nabla_{x}f(x_{t-1},\hat{\theta_{t}})|| \right) \\
&\leq \sum_{t=2}^{T} \left(C_{\eta,\lambda}||x_{t-1}-x_{t}^{\ast}|| + \eta ||\nabla_{x}f(x_{t-1},\theta_{t}) - \nabla_{x}f(x_{t-1},\hat{\theta_{t}})|| \right) \\
&\leq \sum_{t=2}^{T} \left(C_{\eta,\lambda}||x_{t-1}-x_{t}^{\ast}|| + \eta C_{\theta}||\theta_{t}-\hat{\theta_{t}}|| \right) \\
&\leq C_{\eta,\lambda} \sum_{t=2}^{T} ||x_{t-1}-x_{t-1}^{\ast}|| + C_{\eta,\lambda} \sum_{t=2}^{T} ||x_{t-1}^{\ast}-x_{t}^{\ast}|| + \eta C_{\theta} \sum_{t=2}^{T} ||\theta_{t}-\hat{\theta_{t}}|| \\
&\leq C_{\eta,\lambda} ||x_{1}-x_{1}^{\ast}|| + C_{\eta,\lambda}\sum_{t=2}^{T} ||x_{t}-x_{t}^{\ast}|| + C_{\eta,\lambda} \sum_{t=2}^{T} ||x_{t-1}^{\ast}-x_{t}^{\ast}|| + \eta C_{\theta} \sum_{t=2}^{T} ||\theta_{t}-\hat{\theta_{t}}||
\end{align*}
\noindent where the second inequality follows from Lemma \ref{lem:GradCont}, the third inequality follows from the Lipschitz continuity in $\theta$ of the $x$-gradients, the fourth inequality from the triangle inequality, and the last inequality from the norm being nonnegative. Subtracting the second term on the last line from both sides, because $C_{\eta, \lambda} <1,$ we see that

\begin{equation}
\sum_{t=2}^{T} ||x_{t}-x_{t}^{\ast}|| \leq \frac{C_{\eta,\lambda}}{1-C_{\eta,\lambda}} \left(||x_{1}-x_{1}^{\ast}|| + \sum_{t=1}^{T-1} ||x_{t}^{\ast}-x_{t+1}^{\ast}|| \right) + \frac{C_{\theta} \eta}{1-C_{\eta,\lambda}} \sum_{t=2}^{T} ||\theta_{t}-\hat{\theta_{t}}||
\end{equation}

\noindent Plugging this estimate into Equation \ref{eqn:initGradEst} and some minor algebra completes the proof.

\subsection{Proof of Theorem \ref{thm:ExpertLearning}} \label{sec:ExpertProof}
The basic idea of the proof is to follow the basic regret calculation for the expert learning algorithm, and then use the result of Theorem \ref{thm:Main}. The expert learning calculation is routine and adapted from \cite{cesa2006prediction} for the sake of completeness. Assume that the models are indexed from 1 to $N$. Then for model $i,$, define $L_{t}^{i} = \sum_{k=1}^{t}f_{k}(x_{k}^{i})$ and

$$W_{t} = \sum_{i=1}^{N} w_{i,0}e^{-\gamma L_{t}^{i}}$$

From this, it follows that properties of the logarithm that

\begin{equation} \label{eqn:logLow}
\ln W_{T} \geq  \left( \ln \max_{1 \leq i \leq N}w_{i,0}e^{-\gamma L_{T}^{i}} \right) = -\gamma \left(\min_{1 \leq i \leq N} \left(L_{T}^{i} + \frac{1}{\gamma} \ln \frac{1}{w_{i,0}} \right) \right) =-\gamma \left(\min_{1 \leq i \leq N} \left(L_{T}^{i} + \frac{1}{\gamma} \ln N \right) \right)
\end{equation}

By logarithm properties, we also have that $\ln W_{T} = \ln W_{1} + \sum_{i=1}^{T-1} \ln \left(\frac{W_T}{W_{T-1}} \right).$ It is not hard to see by the definition of $w_{i,t},$ we have

\begin{equation} \label{eqn:logSimp}
\ln \left(\frac{W_{t}}{W_{t-1}} \right) = \ln \left( \sum_{i=1}^{N}w_{i,t} e^{-\gamma f_{t}(x_{t}^{i})} \right)
\end{equation}

Note that the sum on the left is an expectation of the random variable $e^{-\gamma f_{t}(x_{t}^{i})},$ so Hoeffding's and Jensen's inequalities gives

\begin{equation} \label{eqn:HoeffJen}
\ln \left(\frac{W_{t}}{W_{t-1}} \right) \leq - \gamma \left(\sum_{i=1}^{N}w_{i,t} f_{t}(x_{t}^{i}) \right) + \frac{\gamma^{2}D^{2}}{8} \leq -\gamma f_{t} \left(\sum_{i=1}^{N} w_{i,t}x_{t}^{i} \right)+ \frac{\gamma^{2}D^{2}}{8} \leq -\gamma f(x_{t}) + \frac{\gamma^{2}D^{2}}{8}
\end{equation}

Combining Equations \ref{eqn:logLow}, \ref{eqn:logSimp}, and \ref{eqn:HoeffJen}, we have

\begin{equation} \label{eqn:wtBnd}
-\gamma \left(\min_{1 \leq i \leq N} \left(L_{T}^{i} + \frac{1}{\gamma} \ln N \right) \right) \leq \ln W_{T} \leq -\gamma \sum_{t=1}^{T} f_{t}(x_{t}) + \frac{\gamma^{2} TD^{2}}{8}
\end{equation}

Cancelling a $\gamma$ on both sides and rearranging Equation \ref{eqn:wtBnd} gives

\begin{equation} \label{eqn:ExpBoundPre1}
\sum_{t=1}^{T} f_{t}(x_{t}) - \min_{i} \left(\sum_{t=1}^{T} f_{t}(x_{t}^{i}) \right) \leq \frac{T \gamma D^{2}}{8} -\frac{1}{\gamma} \ln N
\end{equation}

Routine calculus minimizes the right hand side of Equation \ref{eqn:ExpBoundPre1} by setting $\gamma = \sqrt{8/(TD^{2}}$ to get

\begin{equation} \label{eqn:ExpBound1}
\sum_{t=1}^{T} f_{t}(x_{t}) - \left(\sum_{t=1}^{T} f_{t}(x_{t}^{i}) \right) \leq \frac{D \sqrt{2T}}{4}(1+ \ln N)
\end{equation}

for every $i.$ Since each model follows its own version of predictive online gradient descent, each model has its own bound according to Theorem \ref{thm:Main}, namely:

\begin{equation} \label{eqn:ExpBound2}
\sum_{t=1}^{T} f_{t}(x_{t}^{i}) - \left(\sum_{t=1}^{T} f_{t}(x_{t}^{\ast}) \right) \leq \frac{G||x_{1}-x_{1}^{\ast}||}{1-C_{\eta,\lambda}} + \frac{GC_{\eta,\lambda}}{1-C_{\eta,\lambda}}\mathcal{P}^{\ast} 
+ \frac{G\eta C_{\theta}}{1-C_{\eta,\lambda}} \min_{i} P^{\theta}_{i}
\end{equation}

Combining Equations \ref{eqn:ExpBound1} and \ref{eqn:ExpBound2} for the $i$ with the smallest bound for Equation \ref{eqn:ExpBound2} completes the proof.
\end{appendix}
\end{document}